\documentclass[11pt]{article}
\usepackage{amsmath,amsthm,amssymb,fullpage,enumerate,cleveref,url,titlesec}

\newtheorem{theorem}{Theorem}[section]

\newtheorem{definition}[theorem]{Definition}

\newtheorem{question}[theorem]{Question}
\newtheorem{proposition}[theorem]{Proposition}
\newtheorem{conjecture}[theorem]{Conjecture}

\usepackage{todonotes}
\usepackage{xcolor}
\usepackage{relsize}
\usepackage{comment}
\usetikzlibrary{positioning}

\newcommand{\Q}{\mathbb{Q}}

\newcommand{\Z}{\mathbb{Z}}
\newcommand{\N}{\mathbb{N}}

\DeclareMathAccent{\widehat}{\mathord}{largesymbols}{"62}

\title{Monochromatic Sums and Products over $\mathbb{Q}$}
\author{
Ryan Alweiss
\thanks{Department of Pure Mathematics and Mathematical Statistics and Trinity College, University of Cambridge.
Email: {\tt ra699@cam.ac.uk}.
Research supported by an NSF Mathematical Sciences Postdoctoral Fellowship.}\\
}

\begin{document}
\maketitle


\begin{abstract} 

Hindman's finite sums theorem states that in any finite coloring of the naturals, there is an infinite sequence so that all of its finite subset sums are the same color. In 1979, Hindman showed that there is a finite coloring of the naturals so that no infinite sequence has all of its pairwise sums and pairwise products the same color.  Hindman conjectured that for any $n$, a finite coloring of the naturals contains $n$ numbers all of whose subset sums and subset products are the same color.  In this paper we prove the version of this statement where we color the rationals instead of the integers.  In other words, we show that the pattern $\{ \sum_{i \in S}x_i, \prod_{i \in S}x_i \}$, where $S$ ranges over all nonempty subsets of $[n]$, is partition regular over the rationals. \end{abstract}
	
\section{Introduction}~\label{section:introduction}

One of the oldest results in Ramsey theory is Schur's theorem \cite{schur}, from 1916.  In order to prove that Fermat's last theorem is false in the integers mod $p$, Schur proved that in any finite coloring of $\mathbb{N}$, there are some $x$ and $y$ so that $x,y,x+y$ all receive the same color, i.e. that the pattern $\{x,y,x+y\}$ is \emph{partition regular}.  Just over a decade later, van der Waerden proved that in a finite coloring of $\mathbb{N}$, there are arbitrarily long monochromatic arithmetic progressions.  A few years later, Rado \cite{rado} proved a far-reaching generalization of the theorems of Schur and van der Waerden, fully characterizing which linear systems of equations are partition regular.  One important corollary of Rado's theorem is the so-called Folkman's theorem, which generalizes Schur's theorem.  Folkman's theorem states that in any finite coloring of $\mathbb{N}$, for any $n \ge 2$ there are $x_1, \cdots, x_n$ so that $\sum_{i \in S}x_i$ are all the same color, for any nonempty subset $S$ of $[n]$.  Note that $n=2$ case is just Schur's result. In 1974, Hindman \cite{fsum} proved an infinitary version of Folkman's theorem.  Hindman's Theorem states that in any finite coloring of $\mathbb{N}$, some color class contains an \emph{infinite} sequence and all its finite subset sums. This is about as general of a statement as one can possibly hope for, as far as infinitary Rado is concerned.  In particular, the infinitary versions of the generalizations of Folkman's theorem that follow from Rado's theorem are false. For instance, one cannot also ask for terms of the form $2x_i+x_j$ to be the same color (see \cite{imre1}, \cite{imre2}). While the partition regularity of linear equations is well understood, even in the infinitary case, very little is known about the partition regularity of polynomial equations in general (see e.g. \cite{bergelson}, \cite{blm}, \cite{logic}).

However, there is one class of polynomial equations for which partition regularity is very well understood.  By considering powers of $2$, one can easily characterize the product versions of these statements from the sum versions (equivalently, one can consider partition regularity for linear systems in $\log(x_i)$).  For instance, one can quickly deduce the product version of Folkman's theorem. In any finite coloring of $\mathbb{N}$, for any $n \ge 2$ there are $x_1, \cdots, x_n$ so that $\prod_{i \in S}x_i$ are all the same color, for any nonempty subset $S$ of $[n]$.  

Combining addition and multiplication, which is needed to make progress on the partition regularity of polynomials in general, is far more difficult. Perhaps the most natural questions of this form concern the simplest common generalization of the partition regularity of linear equations and their multiplicative versions. 

Over $\mathbb{N}$, Hindman showed \cite{pairwise} that one cannot even ask for an infinite sequence all of whose \emph{pairwise} sums and products are monochromatic.  Thus, the natural common generalization of Hindman's Theorem and its multiplicative version is false.  This has implications in the theory of ultrafilters; from this it follows that there is no ultrafilter $p \in \beta \mathbb{N}$ with $p+p=p \cdot p$ (\cite{book}, Corollary 17.17).  Recently, Hindman, Ivan, and Leader gave a new construction \cite{maria} of a coloring without such an infinite sequence and made substantial progress toward disproving the infinitary version of the same statement over $\mathbb{Q}$.

In the 1970's, Hindman \cite{hindman} asked about ``the natural finite version of the main sums and products problem", i.e. the natural common generalization of the additive and multiplicative forms of Folkman's theorem.  He has repeated this conjecture on a number of occasions (see e.g. Question 17.18 of \cite{book} and Problem 6.5 of \cite{bropen}, where he states that he is ``absolutely certain that it is a fact").

\begin{conjecture}~\label{hindman} For any $n \ge 2$, if $\mathbb{N}$ is colored in finitely many colors, there exist $x_1, \cdots, x_n$ such that all the numbers $\sum_{i \in S}x_i$ and $\prod_{i \in S}x_i$, for nonempty $S \subset [n]$, are the same color.
\end{conjecture}

Hindman \cite{npc} has also conjectured the weaker form of this, where we color $\mathbb{Q}$ instead of $\mathbb{N}$.


\begin{conjecture}~\label{qhindman} For any $n \ge 2$, if $\mathbb{Q}$ is colored in finitely many colors, there exist $x_1, \cdots, x_n$ such that all the numbers $\sum_{i \in S}x_i$ and $\prod_{i \in S}x_i$, for nonempty $S \subset [n]$, are the same color.
\end{conjecture}


~\Cref{hindman} is one of the most important and longstanding conjectures in partition regularity, and very little is known about it.  In particular, what Hindman \cite{hindman} calls ``the simplest special case" of $n=2$ (i.e. the partition regularity of $\{x,y,x+y,xy\}$ over $\mathbb{N}$\footnote{More precisely, Hindman refers to the $2$-color version of $n=2$ as the ``simplest special case".}), has been highlighted several times since Hindman made some numerical computations for it in his original paper, settling it when $2$ colors are used and giving a lower bound when $3$ colors are used.  It is still open (see e.g. \cite{green}, \cite{hindman2}).  Moreira \cite{moreira} made substantial progress on this problem, showing that $\{x,x+y,xy\}$ is partition regular over $\mathbb{N}$. Over $\mathbb{Q}$, the special case $n=2$ was settled recently in an exciting work of Bowen and Sabok \cite{bs}, who proved that $\{x,y,x+y,xy\}$ is partition regular over $\mathbb{Q}$.  Before \cite{bs}, there had also been previous work on the partition regularity of $\{x,y,x+y, xy\}$ over fields (\cite{cil}, \cite{greensand}, \cite{shk}). By a standard compactness argument (see \cite{bs}), the partition regularity of a polynomial pattern over arbitrary fields follows from its partition regularity over $\mathbb{Q}$, and so \cite{bs} subsumes all of this previous work. For $n>2$, however,~\Cref{qhindman} was open. Over both the setting of fields like $\mathbb{Q}$ and the setting of $\mathbb{N}$, the existing results were all very specialized to the case $n=2$, and there were no nontrivial results about what happens in the general case.

In this paper, we completely settle~\Cref{qhindman}.

\begin{theorem} For any $n \ge 2$, if $\mathbb{Q}$ is colored in finitely many colors, we can find some $x_1, \cdots, x_n \neq 0$ such that all the numbers $\sum_{i \in S}x_i$ and $\prod_{i \in S}x_i$, for nonempty $S \subset [n]$ are the same color.
\end{theorem}

In addition to proving a much more general result, our proof has several key differences from that of Bowen and Sabok.  One important difference is that our proof does not use the result of Moreira as a black box.  Instead, it is similar in spirit to the author's proof of Moreira's result \cite{alweiss}.  Furthermore, in contrast to their proof, and like \cite{alweiss}, our proof obtains explicit bounds.  Additionally, like in \cite{alweiss} and unlike the works of Moreira and Bowen-Sabok, we use the polynomial van der Waerden theorem.  We believe that something like this is necessary to resolve~\Cref{hindman}, as explained in Conjecture $4.1$ of \cite{alweiss}.

\section{Preliminaries}

Throughout this paper, the notation $M \gg M'$ means that $M$ is a sufficiently large function of $M'$.  This notation will be useful for us, because we will often need a large ``reservoir" to ensure that we can apply standard Ramsey theoretic results. For convenience, we define the following notations:

\begin{definition}
The size $s(x)$ of a rational number $x=\frac{a}{b}$ with $\gcd(a,b)=1$ is $\max(|a|, |b|)$.
\end{definition}

\begin{definition} A \emph{good polynomial} $P(x_0; x_1, \cdots, x_n)$ is a rational linear combination of $x_0,\ldots,x_n$ with nonzero $x_0$ coefficient. 

The \emph{size} $s(P)$ of a good polynomial $P(x_0; x_1, \cdots, x_n)=c_0(x_0+\sum_{i=1}^{n}\frac{c_i}{c_0}x_i)$ is the maximum of $s(c_0), s(c_1/c_0), \cdots, s(c_n/c_0)$.  If $c_0=1$, we say that $P$ is \emph{monic}.

For a good polynomial $P$, we also define the notation $P(x_0):=P(x_0, 0, \cdots, 0)=c_0x_0$, and in general $P(x_0; x_1, \cdots, x_j):=P(x_0; x_1, \cdots, x_j, 0, \cdots, 0)=\sum_{i=0}^{j}c_ix_i$.
\end{definition}

The size condition is a technical condition to make our proof finitary.  It might be more intuitive a priori to define the size of a polynomial to be the maximum of the $s(c_i)$.  For technical reasons that will become apparent later, it turns out that this definition is much cleaner.  In any case, the definition of size $\max(c_0, \frac{c_i}{c_0})$ that we use and the naive size $\max(c_i)$ are readily seen to be equivalent up to a quadratic, i.e. they are at most quadratic functions of each other.

In order to prevent division by $0$, we could work over $\mathbb{Q}^{+}$.  In that setting, the same proof we present here would go through, with minor adjustments to avoid sign issues.  In any case, partition regularity over $\mathbb{Q} \setminus \{0\}$ for any of the patterns we are concerned with here is readily seen to be equivalent to partition regularity over $\mathbb{Q}^{+}$.  Indeed it is impossible for two numbers $x_1$ and $x_2$ and their product $x_1x_2$ to all be negative, so if all of the negative numbers are their own color, then we will not have a monochromatic pattern in this color.  In the exact same way, partition regularity over $\mathbb{Z}$ for these patterns is equivalent to partition regularity over $\mathbb{N}$.  For the sake of elegance and to be consistent with notations in the existing literature, we choose to present the proof over $\mathbb{Q}$, glossing over this minor technical detail and implicitly assuming variables are nonzero when necessary; there will always be more than enough freedom to pick them so that this is the case.

One key ingredient in our proof is the \emph{polynomial van der Waerden theorem}, originally due to Bergelson and Leibman \cite{bl2} (see also Walters \cite{walters} for a combinatorial proof).  

\begin{theorem}[Polynomial van der Waerden] Given a finite coloring of $\mathbb{Z}^{\ell}$, and some integer valued polynomials $p_1, \cdots, p_k$ with $p_i(0)=0$, for any vectors $v_1, \cdots, v_k$ there exist $n \in \mathbb{Z}$ and $u \in \mathbb{Z}^{\ell}$ so that $u+p_i(n)v_i$ for $i \le k$ are all the same color.  
\end{theorem}	

The \emph{multidimensional polynomial van der Waerden theorem} (which follows for instance from \cite{bl2}, Theorem $B$) also guarantees that $u+P(p_1(n)v_1, \cdots ,p_k(n)v_n)$ are the same color, for any linear combination $P$ of $p_1(n)v_1, \cdots, p_k(n)v_k$ with $s(P) \le M$.  An easy scaling argument shows the same conclusion holds in $\mathbb{Q}^{\ell}$ or $\mathbb{Q^{+}}^{\ell}$.

\begin{theorem}[Multidimensional Polynomial van der Waerden]\label{thm:mpvdw}
	Given a finite coloring of $\mathbb{Q}^{\ell}$, and some integer valued polynomials $p_1, \cdots, p_k$ with $p_i(0)=0$, for any vectors $v_1, \cdots, v_k$ there exist $n \in \mathbb{Z}$ and $u \in \mathbb{Q}^{\ell}$ so that $u+\sum c_ip_i(n)v_i$ for $i \le k$ are all the same color, as the $c_i$'s range over all rational numbers with size at most $M$.
\end{theorem}

We will use this result extensively in this paper, but we will only ever need the case where $p_i$ are monomials.  In this theorem, by compactness we can always take $u$ and $n$ to be bounded by a function of the number of colors in the initial coloring, the $p_i$, and value of $M$. In particular, they do not depend on the choice of the $v_i$.  

In our proof, we will often be applying the theorem on product colorings up to a certain size.  Given some coloring $\chi: \Q \to [n]$, we will use $\chi^M$ to denote \emph{the product coloring of $\chi$ of size $M$}, the coloring $\chi^M: \Q \to [n]^{|\{m: s(m) \le M\}|}$, where $\chi^M(q)=(\chi(mq))_{s(m) \le M}$.

\section{The $n=3$ case of the main lemma}

Before stating and proving our main lemma, we will prove the $n=3$ case, which captures all of the main ideas.  

\begin{proposition}

\label{prop:three}
For any finite coloring $\chi$ of $\mathbb{Q}$ and any size bound $M$, there exist $a,b,c \in \mathbb{Q}$ so that for any good $P$ with $s(P) \le M$:

\begin{enumerate}
\item The color of $P(a;b,c,bc)$ is the same as that of $P(a)$.
\item The color of $P(ac;b)$ is the same as that of $P(ac)$.
\item The color of $P(b;c)$ is the same as that of $P(b)$.
\end{enumerate}

\end{proposition}

The reader may think of $a \gg b \gg c$ and in fact our proof can easily be modified to guarantee that.  Before presenting the proof of this proposition, we will say a little bit about its structure.  The proof takes the form of an algorithm that finds this pattern, albeit an extremely slow one.  It repeatedly uses the polynomial van der Waerden theorem to find and modify (``update") the variables $a,b,c$ so that the above conditions hold.

We will start by finding $b=b(0)$ and $c=c(0)$ so that the third condition, or third ``bullet point" holds.  The values of $a,b,c$ will change throughout the course of the proof as a function of time $t$, so this $b(0)$, $c(0)$ will not be our final $b,c$.  In particular, we will often ``shift" a variable by adding or subtracting some polynomials of the ``smaller" variables appearing later in the alphabet.  For instance, think about an operation like $b(1)=b(0)+3c(0)$ or $c(1)=5c(0)$.  We call these operations the \emph{shift} $b \leftarrow b+3c$ and the \emph{scaling} $c \leftarrow 5c$.  This notation $x \leftarrow y$ is used in ``pseudocode", which appears in many papers in theoretical computer science. Again, it is helpful to think of our process as an algorithm, and the variables $a,b,c$ evolving in time, with $a(t), b(t), c(t)$ being their values at time $t$.  We will not always explicitly write out the time $t$, instead introducing notation to describe the values of $a(t+1), b(t+1), c(t+1)$ in terms of $a(t), b(t), c(t)$.

\begin{definition}
	Given some choice of variables $(a,b,c)=(a(t),b(t),c(t))$, if we define the variables $a(t+1)=f(a(t),b(t),c(t))$, $b(t+1)=g(a(t),b(t),c(t))$, $c(t+1)=h(a(t),b(t),c(t))$, we refer to these operations as the ``updates" $a \leftarrow f(a,b,c)$, $b \leftarrow g(a,b,c)$, $c \leftarrow h(a,b,c)$. In general, we say $p(a,b,c) \leftarrow q(a,b,c)$ if $p(a(t+1),b(t+1),c(t+1))=q(a(t),b(t),c(t))$.  \end{definition}

For us, $f,g,h,p,q$ will always be polynomials.  In fact, our updates will be of a specific kind.  We rigorously define the aforementioned ``shifting" and ``scaling", two special kinds of operations, which will be the only kind of updates we ever make.  A \emph{shift} is an operation of the form $a \leftarrow a+f(b,c)$ or $b \leftarrow b+f(c)$, where we define $a(t+1)=a(t)+f(b(t),c(t))$ or $b(t+1)=b(t)+f(c(t))$ for some function $f$.  A \emph{scaling} is an operation of the form $a \leftarrow ka$, $b \leftarrow kb$, $c \leftarrow kc$ for a constant $k$, where we define $a(t+1)=ka(t)$, $b(t+1)=kb(t)$, or $c(t+1)=kc(t)$.

By convention, unspecified variables do not change, so if we say for instance the update $b \leftarrow g(a,b,c)$, $c \leftarrow h(a,b,c)$, this is the same as the update $a \leftarrow a$, $b \leftarrow g(a,b,c)$, $c \leftarrow h(a,b,c)$.  This matches the syntax of ``pseudocode" and of most programming languages.

Furthermore, in a slight abuse of notation, the value of the parameter $M$ will itself ``update" downward throughout the proof.  If one insists on being formal, we will define for any $M$ some sequence $M_0 \gg M_1 \gg M_2=M$ with foresight that grows sufficiently quickly, so that initially we will have $M_0$ as a size parameter, and then $M_1$, and so on, and finally $M_r=M$.  In algorithmic language, we initialize $M=M_0$ and then if $M=M_i$ we perform the update $M \leftarrow M_{i+1}$.  We can make $M_r$ as large as we like by choosing $M_0$ appropriately. We will not explicitly write out all of the $M_i$, instead proving intermediate results that hold with arbitrary choices of $M$ by handling the conditions of Proposition~\ref{prop:three} one at a time.  By renaming $M$ along the way, the notation will be a little bit more convenient, especially when we state and prove the most general case of the lemma.  Really we are considering a sequence of $M_i$, chosen with foresight.  All $M_i$ are sufficiently large functions of the number of colors of the initial coloring, and so we may suppress the dependence on the number of colors in the original coloring.  We now prove~\Cref{prop:three}.

\begin{proof}
    Let $\chi$ denote the initial coloring of $\mathbb{Q}$.  On the first step, we will find some initial $b=b(0)$ and $c=c(0)$ so that the third condition holds (one can either set $a(0)$ arbitrarily, or not define it).  The notation may be slightly intimidating, but all we are doing is applying the ordinary van der Waerden theorem to the appropriate product coloring.
            
    Consider the auxiliary product coloring of $\chi$ of some large size $M=M_0$, chosen with foresight: $$\chi_2(x):=\chi^{M}(x)=(\chi(Kx))_{K \in \mathbb{Q}, s(K) \leq M}.$$
     
    By van der Waerden's theorem, which is of course a special case of Theorem~\ref{thm:mpvdw}, we can find $b=b(0)$ and $c=c(0)$ so that $\chi_2(b+P(c))=\chi_2(b)$ for all linear combinations $P$ with $s(P) \le M$.
        
    Now, we claim that in the initial coloring $\chi$, the color of $P(b;c)$ will be the same as that of $P(b)$, if $s(P) \le M$.  This is because if $P(b;c)=k_1b+k_2c=k_1(b+\frac{k_2}{k_1}c)$ has $s(P) \le M$, then $\frac{k_2}{k_1}$ and $k_1$ have size at most $M$.  Hence, $b+\frac{k_2}{k_1}c=\frac{k_1b+k_2c}{k_1}$ will have the same color in $\chi_2$ as $b=\frac{k_1b}{k_1}$. Since $k_1$ also has size at most $M$, by the construction of $\chi_2$, we have that $k_1b+k_2c$ has the same color as $k_1b$ in the initial coloring $\chi$. Thus, for arbitrary choices of $M$, we have found $b,c$ that satisfy the third condition of Proposition~\ref{prop:three}.

We now aim to satisfy the first two conditions.  For technical reasons, we want our condition (modulo size constraints) to be closed under shifting $a$ by $\frac{b}{c}$, i.e. taking the update $a \leftarrow a+\frac{b}{c}$ and thus the update $ac \leftarrow ac+b$.  This sort of transformation ends up being necessary to handle terms like $P(ac; b)$.  This is also the main technical reason that our proof does not work in $\N$, since $\frac{b}{c} \notin \N$ in general.  Similarly, we need our operation to be closed under taking $a \leftarrow a+b$, $a \leftarrow a+c$, and $a \leftarrow a+bc$ to handle $P(a;b)$, $P(a;c)$, and $P(a;bc)$ respectively. So, we will actually end up showing something a little stronger.

\begin{enumerate}
\item The color of $P(a;\frac{b}{c},b,c,bc)$ is the same as that of $P(a)$.
\item The color of $P(ac;b, bc, c^2, bc^2)$ is the same as that of $P(ac)$.
\item The color of $P(b;c)$ is the same as that of $P(b)$.
\end{enumerate}

Recall that we have already defined some $b,c$ that satisfy the third condition with size parameter $M=M_0$, which we can take to be arbitrarily large.  We initialize the value $M'$ to be some parameter $M_1$, which can be taken to be as large as desired, but which is small compared to $M_0$. 

We now describe how to find an $a=a(1)$ that satisfies the first condition.  We will apply multidimensional polynomial van der Waerden (Theorem~\ref{thm:mpvdw}) in $\mathbb{Z}^4$.  Consider the coloring $\chi_3$ of $\mathbb{Z}^4$ given by $$\chi_3(q_1,q_2,q_3,q_4)=\chi^{M'}(q_1 \frac{b}{c}+q_2b+q_3c+q_4bc).$$

With 
$(p_1,p_2,p_3,p_4)=(n,n^2,n,n^3)$, multidimensional polynomial van der Waerden gives us some $u=(u_1,u_2,u_3,u_4) \in \mathbb{Z}^4$ and some $k \in \mathbb{N}$ such that the vectors $$(u_1,u_2,u_3,u_4)+\ell_1(k,0,0,0)+\ell_2(0,k^2b,0,0)+\ell_3(0,0,kc,0)+\ell_4(0,0,0,k^3bc)$$ for $\ell_i$ with $s(\ell_i) \le M' \ll M$ receive the same color under $\chi_3$, and hence the numbers

$$u_1 \frac{b}{c}+u_2b+u_3c+u_4bc+\ell_1k\frac{b}{c}+\ell_2k^2b+\ell_3kc+\ell_4k^3bc,$$ for $\ell_i$ with $s(\ell_i) \le M' \ll M$, all receive the same color under $\chi^{M'}$.  The number of colors used in $\chi_3$ is some function of $M'$ but is still much smaller than $M$.  The polynomial van der Waerden theorem guarantees that the $u_i$'s and $k$ are functions of $M'$ and are thus still much smaller than $M$. 

Defining $a=u_1\frac{b}{c}+u_2b+u_3c+u_4bc$ and performing the update $b \leftarrow k^2b$ and $c \leftarrow kc$ (i.e. setting $b(1)=k^2b(0)$ and $c(1)=kc(0)$, and now using $b,c$ to mean $b(1),c(1)$), we have that the color of $a+\ell_1\frac{b}{c}+\ell_2b+\ell_3c+\ell_4bc$ in $\chi^{M'}$ is the same as the color of $a$ as long as $s(\ell_1), s(\ell_2), s(\ell_3), s(\ell_4) \le M'$.  So $P(a;\frac{b}{c},b,c,bc)$ is the same color as $P(a)$ in $\chi^{M'}$ as long as $s(P) \le M$ and $P$ is monic.

Thus the color of $K(a+\ell_1\frac{b}{c}+\ell_2b+\ell_3c+\ell_4bc)$ in the original coloring $\chi$ is the same as the color of $Ka$ as long as $s(K) \le M'$ and $s(\ell_i) \le M'$ for $1 \le i \le 4$. In other words, $P(a;\frac{b}{c},b,c,bc)$ is the same color as $P(a)$ in $\chi$ whenever $s(P) \le M$. Clearly, this scaling of $b,c$ preserves the third condition, with $M$ replaced by a slightly smaller size bound that comfortably exceeds $M'$.

Hence we have found an $a,b,c$ so that the first and third conditions are satisfied with parameter $M'=M_1$.  Again by setting the starting parameters appropriately we can make it as large as possible, and in what follows we refer to $M_1$, which we have been calling $M'$, as $M$.  We will refer to $M_2$ as $M'$.  In other words, we perform the update $M \leftarrow M_1$, $M' \leftarrow M_2$.

It remains to satisfy the second condition.  We will again apply Theorem~\ref{thm:mpvdw} in $\mathbb{Z}^4$ to an appropriate auxiliary coloring.  This time, we use the coloring
$$\chi_4(q_1,q_2,q_3,q_4)=\chi^{M'}((a+q_1 \frac{b}{c}+q_2b+q_3c+q_4bc)c).$$  We factor out the $c$ so that we may shift $a$.

Applying Theorem~\ref{thm:mpvdw} with $(p_1,p_2,p_3,p_4)=(n^2,n^3,n^2,n^4)$ 
yields some $u=(u_1,u_2,u_3,u_4) \in \mathbb{Z}^4$ and some $k \in \mathbb{N}$, both functions of $M'$ and small compared to $M$, such that the numbers $$(a+u_1 \frac{b}{c}+u_2b+u_3c+u_4bc+\ell_1k^2\frac{b}{c}+\ell_2k^3b+\ell_3k^2c+\ell_4k^4bc)c,$$ all receive the same color under $\chi^{M'}$, when the $\ell_i$'s have size at most $M'$.  Shift $a$, performing the update $a \leftarrow a+u_1 \frac{b}{c}+u_2b+u_3c+u_4bc$, so that the numbers $$(a+\ell_1k^2\frac{b}{c}+\ell_2k^3b+\ell_3k^2c+\ell_4k^4bc)c$$ are all the same color under $\chi^{M'}$ whenever the $\ell_i$'s have size at most $M'$.

Finally, perform a scaling update $a \leftarrow a/k$, $b \leftarrow k^2b, c \leftarrow kc$, so that now the numbers $$(a+\ell_1\frac{b}{c}+\ell_2b+\ell_3c+\ell_4bc)c$$ all receive the same color in $\chi^{M'}$.  So $P(ac;b,bc,c^2,bc^2)$ has the same color as $P(ac)$ in $\chi^{M'}$ whenever $P$ is monic. Because $k$ is a function of $M'$ and is much less than $M$, this preserves the first and third conditions, again with a slight reduction in the size guarantee that still comfortably exceeds $M'$. By the definition of $\chi^{M'}$, if $s(K) \le M'=M_2$, then $K(a+\ell_1\frac{b}{c}+\ell_2b+\ell_3c+\ell_4bc)c$ has the same color as $Kac$, as desired.  In other words $P(ac;b,bc,c^2,bc^2)$ has the same color as $P(ac)$ in $\chi$ for all $s(P) \le M'=M_2$.  So all three conditions are satisfied with parameter $M'=M_2$.  Performing the final update $M \leftarrow M_2$, we have the desired result. \end{proof}

\section{Statement and proof of the main lemma}~\label{section:outline}
In this section, we state and prove the main lemma in its full generality. All of the necessary ideas in the proof are contained in the last section.  The difficulty is purely notational rather than conceptual, and indeed we need to define a little bit of new notation to deal with having more than constantly many variables $x_n, x_{n-1}, \cdots, x_1$ that depend on time $t$.  We rigorously define updates, shifts, and scaling for this general case of $n$ variables.  An \emph{update} $p(x_n, \cdots, x_1) \leftarrow q(x_n, \cdots, x_1)$ means that $p(x_n(t+1), \cdots, x_1(t+1))=q(x_n(t), \cdots, x_1(t))$.  A \emph{shift} is an update of the form $x_i \leftarrow x_i+q(x_{i-1}, \cdots, x_1)$ and a \emph{scaling} is one of the form $x_i \leftarrow kx_i$.  For notational convenience, we will suppress the time dependence.  Again one can think of this as an algorithm being executed by a computer program.  

If $S$ and $T$ are subsets of $[n]$, then we say $S>T$ if all elements of $S$ are larger than all elements of $T$.    In general we will need to consider good polynomials with more than constantly many monomials.  For a collection or family $\textbf{S}$ of sets $S$, we say that $P(x_i \prod_{j \in T}x_j; \prod_{j \in \textbf{S}}x_j)$ is $P(x_i \prod_{j \in T}x_j; \prod_{j \in S_1}x_j, \prod_{j \in S_2}x_j, \cdots )$ over all possible choices $S_1, S_2, \cdots$ of $S \in \textbf{S}$.  The sets $S_i \in \textbf{S}$ need not be disjoint.  Usually we will take $\textbf{S}$ to be all nonempty subsets of some set, typically $[n-1]$.  In a slight abuse of notation, we will also write $P\left(x_n;\frac{\prod_{j \in \textbf{S}}x_j}{\prod_{j \in \textbf{T}}x_j}\right)$ and specify constraints on $S \in \textbf{S}$ and $T \in \textbf{T}$ (generally something like $\emptyset \neq S>T$, where $\{n\}>S$) to mean that the RHS is actually a list of all $S \in \textbf{S}$ and $T \in \textbf{T}$ that satisfy these constraints.  As in the previous section, we also will define with foresight some sequence $M_t$ depending on the time parameter $t$.  Again, we will suppress the time parameter and just refer to $M$ and $M'$, which we will update and which will always be $M_r$ and $M_{r+1}$ respectively for some $r$.  As in the previous section, we will not explicitly write out how large $M_r$ needs to be in terms of $M_{r+1}$, but we will always be able to find $M_{r+1}$ as large as desired by taking a sufficiently large $M_r$, which we can guarantee in turn by taking a sufficiently large $M_{r-1}$, and so on and so forth.

We need a bit more notation.  We will have many $q_i, u_i, \ell_i$ in our applications of the polynomial van der Waerden theorem.  Instead of indexing the $q$, $u$, $\ell$ by $i$'s, we will typically index them by sets like $S$ and $T$.  As is standard, we use the notation $(q_i)_{i \in U}$ to make a vector of $q_i$'s for $i \in U$, so for example $(q_{S,T})_{S>T}$ means a vector of $q$'s, with a coordinate $q_{S,T}$ for every pair $S$ and $T$ of sets with $S>T$.  Of course we use similar notation for $u$ and $\ell$.  Without further ado, we present the statement of the main lemma.

\begin{proposition}
\label{prop:main}
For any finite coloring $\chi$ of $\mathbb{Q}$, any $M$, and any integer $n \ge 2$, there are numbers $x_n, x_{n-1}, x_{n-2}, \cdots, x_1 \in \mathbb{Q}$ so that for each $i \in [n]$ and $T \subseteq [i-1]$ (possibly empty), if $P$ is good with $s(P) \le M$ then the color of $P(x_i \prod_{j \in T}x_j; \prod_{j \in \textbf{S}}x_j)$ is the same as that of $P(x_i \prod_{j \in T}x_j)$, where $\textbf{S}$ consists of all nonempty subsets $S$ of $[i-1]$ so that $S>T$.
\end{proposition}

The statement is perhaps daunting on first glance, so the reader may wish to think about the $n=4$ case.  If $x_4=a, x_3=b, x_2=c, x_1=d$, this case states that if $s(P) \le M$, then 

\begin{enumerate}
\item The color of $P(a;b,c,d,bc,cd,bd,bcd)$ is the same as that of $P(a)$.
\item The color of $P(ac;b)$ is the same as that of $P(ac)$.
\item The color of $P(ad;b,c,bc)$ is the same as that of $P(ad)$.
\item The color of $P(acd;b)$ is the same as that of $P(acd)$.
\item The color of $P(b;c,d,cd)$ is the same as that of $P(b)$.
\item The color of $P(bd;c)$ is the same as that of $P(bd)$.
\item The color of $P(c;d)$ is the same as that of $P(c)$.
\end{enumerate}

The $n=3$ case  was proved in the previous section with $x_3=a, x_2=b, x_1=c$.  The $n=2$ case with $x_2=b$ and $x_1=c$ is the ``third bullet point" from the proof of the $n=3$ case, and is just van der Waerden on the auxiliary product coloring.

The reader can check that in general for $n$ variables, our bulleted list would have $2^{n-1}-1$ items.  While the number of terms is large, the proof strategy is quite similar to what we have already done.  One should think of $x_n \gg x_{n-1} \gg \cdots \gg x_1$, and the proof can be easily modified to guarantee this.  Induction allows us to handle all of the terms that do not involve the ``biggest" variable $x_n$.  Again, this is precisely what we did in the last section, handling the condition involving $P(b;c)$ first, proving the $n=2$ case as the first step of proving the $n=3$ case.  After this, the terms involving the biggest variable are handled.  While this can be done in any order, we find it easiest to think about handling the case where $T$ is the empty set first, i.e. the terms $P(x_n; \prod_{j \in \textbf{S}}x_j)$ where $\textbf{S}$ consists of all nonempty subsets $S$ of $[n-1]$ (equivalently, $\emptyset \neq S<\{n\}$), because it is slightly simpler and does not involve ``scaling" $x_n$.  For $n=3$, this was the ``first bullet point" $P(a;\frac{b}{c},b,c,bc)$.  Then we handle the other terms by one.  As aforementioned, at each step we lose something in the value of the size parameter $M$, but this is fine, since we chose $M_0$ to be large enough with foresight.


We can always take the smallest final value of $M$ to be as large as desired, and certainly large compared to both $n$ and the number of colors in the original coloring $\chi$. Hence, we suppress the dependence on $n$ and on the number of colors in $\chi$, treating these as constants.  In the previous section we proved the $n=3$ case but again suppressed the dependence on the number of colors.  If it makes the argument easier to follow, the reader may think about proving the lemma for $n=100$ where there are $100$ colors.


\begin{proof}
We use induction on $n$ to prove the proposition in general.  By the induction hypothesis, we have found $x_{n-1}, x_{n-2}, \cdots, x_2, x_1$ satisfying the proposition, where $n \ge 4$.

The first order of business is to establish the proposition for the case where $i=n$ and $T=\emptyset$, i.e. for $P(x_n; \prod_{j \in S}x_j)$, where $S<\{n\}$ ranges over all nonempty subsets of $[n-1]$.

Looking ahead to the term $P(x_n \prod_{j \in T}x_j; \prod_{j \in S}x_j)$, we must in fact establish the conclusion of the proposition for expressions of the form

$$P\left(x_n;\frac{\prod_{j \in \textbf{S}}x_j}{\prod_{j \in \textbf{T}}x_j}\right)$$ where $S \in \textbf{S},T \in \textbf{T}$ range so that $\{n\}>S>T$ and $S$ is nonempty.


Because of this, we must also control all terms of the form
\begin{equation}\label{eq:1}
P\!\left(
  x_n \prod_{j \in T'} x_j \,;\,
  \frac{\prod_{j \in \textbf{S}} x_j \prod_{j \in T'} x_j}{\prod_{j \in \textbf{T}} x_j}
\right).
\end{equation}
Here we take $S \in \textbf{S}$ and $T \in \textbf{T}$, with $S$ nonempty, $\{n\}>S>T$, and $T'<\{n\}$ a (possibly empty) subset of $[n-1]$.

While perhaps a little bit notationally daunting, all we are doing is writing down the ``obvious" strengthening of the lemma so that we may perform our ``shifting" 	and ``scaling" operations.  This is exactly what we did in the previous section.  We again stress that there are really no new conceptual ideas in this proof, and that this division is why the proof does not work over $\N$.

We first handle the case where $T'$ is empty. In the last section where $n=3$, this was the ``first bullet point", where we ensured that $P(a;b,c,bc,\frac{b}{c})$ is the same color as $P(a)$.  The astute reader may notice that it is not strictly necessary to do this case first, and later we will show how to handle all $T'$ in any order, without any assumptions on whether $T'=\emptyset$ or not.  Nonetheless, we find it more intuitive to present this simpler case first.

We define the usual auxiliary product coloring.  Here we will use $(q_{S,T})_{\emptyset \neq S>T, \{n\}>S}$ to denote a vector of $q_{S,T}$ with a coordinate for every pair $S,T$ of subsets with $\emptyset \neq S>T$ and $\{n\}>S$.

 $$\chi'\left((q_{S,T}\right)_{\emptyset \neq S>T, \{n\}>S})=\chi^M\left( \sum_{\emptyset \neq S>T, \{n\}>S} q_{S,T} \frac{\prod_{j \in S}x_j}{\prod_{j \in T}x_j}\right).$$ 

We use Theorem~\ref{thm:mpvdw} to pick $x_n$ to be a linear combination $\sum_{\emptyset \neq S>T, \{n\}>S} u_{S,T} \frac{\prod_{j \in S}x_j}{\prod_{j \in T}x_j}$ where $(S,T)$ range with $\emptyset \neq S>T$ and $\{n\}>S$ and $u_{S,T}$ depend on $M$ so that if $s(P) \le M$, then the color in $\chi'$ of $$P\left(x_n;k^{\sum_{j \in S}y_j-\sum_{j \in T}y_j}\frac{\prod_{j \in \textbf{S}}x_j}{\prod_{j \in \textbf{T}}x_j}\right)$$ is the same as that of $P(x_n)$ if $P$ is monic, for some integer $k$ that depends on $M$, and some sufficiently fast-growing sequence $y_j$ (for instance, $y_j=2^j$ suffices), where $S \in \textbf{S}$ and $T \in \textbf{T}$ range with $\emptyset \neq S>T$ and $\{n\}>S$.  Henceforth, we will not explicitly repeat this constraint on $S \in \textbf{S}$ and $T \in \textbf{T}$, but it will be understood in all expressions with $\textbf{S}$ and $\textbf{T}$.

Crucially, the exponent $\sum_{j \in S}y_j-\sum_{j \in T}y_j$ will be positive, which is necessary for the application of polynomial van der Waerden.  Thus, in $\chi$, the color of $$P\left(x_n;k^{\sum_{j \in S}y_j-\sum_{j \in \textbf{T}}y_j}\frac{\prod_{j \in \textbf{S}}x_j}{\prod_{j \in \textbf{T}}x_j}\right)$$ is the same as that of $P(x_n)$ if $s(P) \le M$.

Perform the scaling update $x_j \leftarrow x_jk^{y_j}$ for all $1 \le j<n$, so that the color of $$P\left(x_n;\frac{\prod_{j \in \textbf{S}}x_j}{\prod_{j \in \textbf{T}}x_j}\right)$$ is now the same as that of $P(x_n)$ for $s(P) \le M$. 

  Again, while the notation is a bit intimidating, what we did is exactly the same as the ``first bullet point" of the last section, where $P(a;b,c,bc,\frac{b}{c})$ is the same color as $P(a)$.  We just need to choose the degrees $y_j=2^j$ to make sure that when we are applying the polynomial van der Waerden, the polynomials have a positive degree.

With this out of the way, we now handle the other choices of $T'$ one by one, in an arbitrary order.  Because there are so many terms now, it is useful to rigorously define what ``handling" means.  We say that $T'$ is \emph{$M$-handled} (or sometimes just ``handled" if $M$ is clear from context) if $P\left(x_n\prod_{j \in T'}x_j;\frac{\prod_{j \in \textbf{S}}x_j\prod_{j \in T'}x_j}{\prod_{j \in \textbf{T}}x_j}\right)$ is the same color as $P(x_n\prod_{j \in T'}x_j)$ whenever $s(P) \le M$.  In other words, the expression \eqref{eq:1} corresponding to $T'$ is controlled.

We have already handled $T'=\{ \emptyset \}$.  In other words, if we let $\mathbb{T}$ be the family of handled $T'$, we have shown how to expand it from $\emptyset$ to $\{ \emptyset\}$.  We will handle the remaining choices of $T'$ one by one, performing updates of the form $\mathbb{T} \leftarrow \mathbb{T'} \cup T'$.

All we are doing here is going through the ``bulleted list" as in the last section, but now there are exponentially many items on the list.  The handled $T'$ are just the ``bullet points" we have already dealt with.  So far, we have only dealt with $T'=\{\emptyset\}$, the ``first bullet point" from the last section.  The family $\mathbb{T}$ just corresponds to the bullet points that we have gone through so far, which are labelled with the various $T'$.

Let us consider $T'$ which has not been handled yet.  We want to ensure that

$$P\left(x_n\prod_{j \in T'}x_j;\frac{\prod_{j \in \textbf{S}}x_j\prod_{j \in T'}x_j}{\prod_{j \in \textbf{T}}x_j}\right)$$ is the same color as $P(x_n \prod_{j \in T'}x_j)$ as long as $s(P) \le M' \ll M$.  In other words, $T'$ will be $M'$-handled.  What we do will be completely analogous to, and in fact a straightforward generalization of, what we did in the previous section to deal with the ``second bullet point" $P(ac;b,bc,c^2,bc^2)$ in the $n=3$ case.  Of course, the notation will be a little more complicated.  

The key difference from what we just did, the special case where $T'=\{\emptyset\}$, is that we need to perform a ``scaling" operation on $x_n$ when $T'$ is nonempty.  This is because, to apply the polynomial van der Waerden, we need to factor out $\prod_{j \in T'}x_j$ before shifting $x_n$, and then we need to scale $x_n$ to compensate for the scaling of $\prod_{j \in T'}x_j$.  Again, this is not new.  In the ``second bullet point" $P(ac;b,bc,c^2,bc^2)$ of the previous section we had to factor out $c$ before shifting $a$, and then scale $a$ afterwards to deal with the fact that $c$ was scaled.

As in the last section, we also need to ensure that our earlier conditions are preserved, albeit with a slightly smaller size guarantee.  We proceed to our final application of the polynomial van der Waerden theorem.

We define the auxiliary product coloring, again over all $\emptyset \neq S>T$ and $S<\{n\}$: $$\chi'\left((q_{S,T}\right)_{\emptyset \neq S>T}, S<\{n\})=\chi^M\left( \left( x_n + \sum_{\emptyset \neq S>T, S<\{n\}} q_{S,T} \frac{\prod_{j \in S}x_j}{\prod_{j \in T}x_j} \right) \prod_{j \in T'} x_j \right).$$

Using Theorem~\ref{thm:mpvdw}, we find some $u_{S,T}$'s and some $k$, both depending on $M'$ (and so $ \ll M$) so that, where $y_j=2^j$:

$$ \left(x_n + \sum_{\emptyset \neq S>T, \{n\}>S}u_{S,T} \frac{\prod_{j \in S}x_j}{\prod_{j \in T}x_j}+\sum_{\emptyset \neq S>T, \{n\}>S}k^{\sum_{j \in S}y_j-\sum_{j \in T}y_j+\sum_{j \in T'}y_j} \ell_{S,T} \frac{\prod_{j \in S}x_j}{\prod_{j \in T}x_j} \right) \prod_{j \in T'} x_j$$ are all the same color in $\chi'$ whenever $s(\ell_{S,T}) \le M'$ for all $\emptyset \neq S>T$ and $\{n\}>S$.

Perform the shift $$x_n 
\leftarrow x_n + \sum_{\emptyset \neq S>T, \{n\}>S}u_{S,T} \frac{\prod_{j \in S}x_j}{\prod_{j \in T}x_j}$$ and then the scaling operations $x_j \leftarrow x_j k^{y_j}$ for $j<n$.  Lastly, performing the scaling operation 

$$x_n \leftarrow \frac{x_n}{k^{\sum_{j \in T'}y_j}}.$$

Conceptually, the first shift of $x_n$ removes the $u_{S,T}$ terms, the scaling of $x_j$ for $j<n$ removes the powers of $k$, and the scaling of $x_n$ deals with factors of $k$ on the outside that arise from the $\prod_{j \in T'}x_j$ because of the scaling of $x_j$ for $j \in T'$.

After these updates, we have the condition without the $u_{S,T}$ and powers of $k$:

$$ \left(x_n + \sum_{\emptyset \neq S>T, \{n\}>S}\ell_{S,T} \frac{\prod_{j \in S}x_j}{\prod_{j \in T}x_j} \right) \prod_{j \in T'} x_j$$ are all the same color in $\chi'$ whenever $s(\ell_{S,T}) \le M'$, for all $\emptyset \neq S>T, \{n\}>S$.  By the definition of $\chi'$, this means that $T'$ has now been $M'$-handled.

For $T_1 \in \mathbb{T}$ that have been previously $M$-handled, because $\ell_{S,T}$ and $k$ depend only on $M'$, the terms: $$P\left(x_n\prod_{j \in T_1}x_j;\frac{\prod_{j \in \textbf{S}}x_j\prod_{j \in T_1}x_j}{\prod_{j \in \textbf{T}}x_j}\right)$$ 

will still be handled after these operations, albeit with a smaller size guarantee.  Since $M \gg M'$, this guarantee comfortably exceeds $M'$, since $M$ is chosen to be large enough.  Similarly, the size guarantee for the terms that do not involve $x_n$ shrinks, but is still comfortably at least $M'$ since we chose $M$ to be sufficiently large.  Hence, we $M'$-handle $T'$ , but every previously $M$-handled $T_1 \in \mathbb{T}$ is still $M'$-handled.  Of course, here the size guarantee $M=M_{t-1}$ drops to $M'=M_t$, but is still as large as we desire.  From now on, we use $M$ to refer to $M_t$ and $M'$ to refer to $M_{t+1}$. In other words, we perform the update $M \leftarrow M_t$ and $M' \leftarrow M_{t+1}$.  If $\mathbb{T}$ is the set of handled sets, we also are performing the update $\mathbb{T} \leftarrow \mathbb{T} \cup T'$, as promised.

In this way, we can handle all the remaining terms $P(x_i\prod_{j \in T'}x_j; \prod_{j \in \textbf{S}}x_j)$ in any order, where $S \in \textbf{S}$ ranges over all nonempty subsets $\{n\}>S>T'$.  Thus, at the end, for an $M$ as large as desired, we have $M$-handled the family $\mathbb{T}$ of all possible $T'$.  Furthermore, we still have our desired conditions with size parameter $M$ for the terms that do not involve $x_n$.  Hence, we are done. \end{proof}

\section{Proof of main result}

We will prove the main result in four steps.

\vspace{1em}
\noindent \textbf{Step I:} By Proposition~\ref{prop:main}, for arbitrary $n$, we can find $x_n, \cdots, x_1$ so that $x_i \prod_{j \in T}x_j+P(\prod_{j \in \textbf{S}}x_j)$ is the same color as $x_i \prod_{j \in T}x_j$, whenever $P$ has coefficients $0$ and $1$, and $\textbf{S}$ is the family of subsets $S$ of $[n]$ with $\{i\}>S>T$.  


\vspace{1em}
\noindent \textbf{Step II:} By Ramsey's theorem, we can pass to a subset of the $x_i$ so that the color of $x_{i_1}x_{i_2} \cdots x_{i_k}$ depends only on $k$ as long as $k \le N$ for some large $N$.  We do this by iteratively applying the Ramsey theorem for $k$-element subsets. First, apply the one dimensional Ramsey theorem to the set $X_0 = \{x_i: i \le n\}$ to find a monochromatic set $X_1 \subseteq X_0$. Next apply the two-dimensional Ramsey theorem to the set $X_1^2$ with
the coloring $\chi'(x_i, x_j) = \chi(x_i x_j)$ to find a set $X_2 \subseteq X_1$ with $X_2$ monochromatic according to this coloring. Continuing in this way gives the desired subsequence, and we may assume it is of size at least $N$ if the initial $n$ is large enough.


Renaming variables, we can find $x_N, \cdots, x_1$ so that for all $k \le N$, there is a color $f(k)$ so that $x_i \prod_{j \in T}x_j+P(\prod_{j \in \textbf{S}}x_j)$ has color $f(k)=f(|T|+1)$, where $i \le N$, $S \in \textbf{S}$ ranges so that $S$ and $T$ are disjoint sets of $[i-1]$ so that $S$ is nonempty and $\{i\}>S>T$, and $P$ is any linear combination with coefficients $0$ and $1$ as in the previous step.  In particular, the color of every monomial $x_i\prod_{j \in T}x_j$ will only depend on its degree $|T|+1 \le N$.  Again, we can make $N$ as large as desired by setting $n$ appropriately.

\vspace{1em}
\noindent \textbf{Step III:} 

Given the sequence $x_N, \cdots, x_1$ from the previous step, consider the auxiliary coloring of $[N]=\{1,2, \cdots, N\}$ where $k$ receives the color $f(k)$ of some (and hence, by the previous step, any) $\prod_{j \in T}x_j$ where $|T|=k$ and $T \subset [N]$.  We can take $N$ to be as large as desired.  Applying Folkman's theorem to this coloring, we can find some $a_1, \cdots, a_r$ so that all terms of the form $\sum_{i \in S}a_i$ are monochromatic and in $[N]$.  We can guarantee that $r$ is as large as desired by choosing $N$ appropriately.

By our choice of $x_N, \cdots, x_1$ in the previous step, this means that all products of $\sum_{i \in S}a_i$ many $x_i$ are the same color (say, red) for any choice of $S \subset [r]$.  For minor technical reasons, we will in fact assume that $\sum_{i \in S}a_i<N/100$.  Again this is not an issue as we can set $N$ as large as desired.

%
%

\vspace{1em}
\noindent \textbf{Step IV:} Now, pick (possibly empty) sets $S_i$ of size $a_r-1$ for $1 \le i \le r$, so that $S_{i+1}>S_i$ and $\{\frac{N}{2}\}>S_r$.  In other words, everything in $S_{i+1}$ is larger than everything in $S_i$, and everything in $S_r$ and hence in all of the $S_i$ is less than $\frac{N}{2}$. Since $\sum_{i \in S}a_i<N/100$, certainly such $S_i$ exist.  

For $1 \le i \le r$, let $X_i=x_{N-i}\prod_{j \in S_i}x_j$. Any product of a subset $S$ of the $X_i$ will be a monomial of degree $\sum_{i \in S}a_i$ and so will be red by our application of Folkman's theorem. Any sum of a subset $S=\{X_{i_1}, \cdots, X_{i_s}\}$ of the $X_i$, for $i_1>i_2> \cdots >i_s$ will be of the form $$x_{N-i_s}\prod_{j \in S_{i_s}}x_j+P(x_{N-i_1}\prod_{j \in S_{i_1}}x_j, \cdots, x_{N-i_{s-1}}\prod_{j \in S_{i_{s-1}}}x_j)$$ where $P$ is a polynomial with coefficients $0$ and $1$. In other words, the second term is a finite (and possibly empty) sum of $x_{N-i_t} \prod_{j \in S_{i_t}}x_j$ for $t<s$.

This is red by our assumption. Renaming $X_i$ as $x_i$ (performing the ``update" $X_i \leftarrow x_i$), and renaming the number of variables $r$ as $n$ (performing the ``update" $r \leftarrow n$) finishes the proof.

\section{Discussions and open problems}

We first note that our proof can be extended to give some more patterns.  The argument in this paper straightforwardly also gets $\sum_{i \le m} \prod_{j \in S_i}x_j$ to be monochromatic, as long as all of the elements of $S_i$ are less than all of the elements of $S_{i+1}$. Furthermore, Bowen-Sabok \cite{bs} ask about patterns like $\{x+P(y), y,xy\}$ where $P$ ranges over a finite set of monomials.  This can also be handled with the argument here.  If we consider expressions like $P(x_i\prod_{j \in T}x_j;\prod_{j \in S}x_j, \prod_{j \in S}x_j^2)$ and not just $P(x_i\prod_{j \in T}x_j;\prod_{j \in S}x_j)$, the partition regularity of $\{x, x+y, x+y^2, y, xy\}$ over $\mathbb{Q}$ follows.  By considering slightly more general $(S,T)$ with $\max_{j \in S}j>\max_{j \in T}j$, Hunter \cite{hunter} proved the partition regularity over $\mathbb{Q}$ of sums of distinct products, i.e. $\sum_{S}\prod_{i \in S}x_i$ so that no $i$ appears in more than one $S$.  It is also worth noting that in an earlier paper \cite{alweiss}, the author proved that $\{x,x+y,xy\}$ was partition regular over $\mathbb{N}$ with effective bounds, essentially by considering a more restricted family of $(S,T)$ than the ones we consider here, so that the argument goes through over the integers.


Interestingly, our proof here implicitly also proves that expressions like $\{dx, dx+d^2\}$ are partition regular over $\mathbb{Q}$, and also over $\mathbb{N}$.  To our knowledge, this is an original result, although the proof is simple.  We include a self-contained explicit proof here.

\begin{theorem}
	The pattern $\{dx, dx+d^2\}$ is partition regular.
\end{theorem}

\begin{proof}

Fix two integers $x,d$, both divisible by $N!$ for a sufficiently large $N$. Given a coloring $f$ of the integers, define a new coloring $g$ of the integers by $g(C)=f(dx+Cd^2)$. By Theorem~\ref{thm:mpvdw} with $p(n)=n^2$, we can find some $k$ and $c$ bounded by $N$ so that $g(k)=g(k+c^2)$.  By the definition of $g$, this means that $f(dx+kd^2)=f(dx+(k+c^2)d^2)=f(dx+kd^2+(cd)^2)$. Letting $x/c+kd/c=X$ and $cd=D$, we get $DX, DX+D^2$ are the same color.  We may rename $D,X$ as $d,x$. \end{proof}

A very similar argument shows that $\{Q(d)(X+P_1(d)), \cdots, Q(d)(X+P_k(d))\}$ is partition regular for any monic polynomial $Q$ and any family of $P_i$ with $P_i(0)=0$. One interesting question, which we suspect would help prove~\Cref{hindman}, is whether we can extend this argument to deal with $Q(d)(X+P_i(d))$ where $Q(0)=0$ but $Q$ is not necessarily monic. As aforementioned, in this paper we only ever use the polynomial van der Waerden theorem for monic polynomials.

\begin{question}
	Let $Q, P_1, \cdots, P_k$ be polynomials with $Q(0)=0$ and $P_i(0)=0$ for all $1 \le i \le k$.  Is the pattern $\{Q(d)(x+P_1(d)), \cdots, Q(d)(x+P_k(d))\}$ necessarily partition regular?
\end{question}

We specifically conjecture the case where $Q(d)=d^2+d$, $k=2$, $P_1(d)=0$, $P_2(d)=d$.

\begin{conjecture}
	The pattern $\{(d^2+d)x, (d^2+d)(x+d)\}$ is partition regular.
\end{conjecture}

It might be fruitful to come up with a different proof of the polynomial van der Waerden theorem, one not so reliant on the presence of short cycles in Cayley graphs, which is necessary for the usual color focusing arguments.  As such, we propose another conjecture.

\begin{conjecture}
	Let $G$ be the ``square-difference graph" on $\mathbb{Z}$, i.e. the graph in which $x$ and $x+d^2$ are adjacent for all $x,d \in \mathbb{Z}$ with $d \neq 0$.  Then $G$ has induced subgraphs of arbitrarily large girth and arbitrarily large chromatic number.
\end{conjecture}

Another interesting question is whether we can have a version of Theorem~\ref{thm:mpvdw} where we allow negative exponents on $d$.  We feel that for instance the following statement should be true.

\begin{conjecture}
	The pattern $\{x, x+d, x+\frac{1}{d}, x+d+\frac{1}{d}\}$ is partition regular over $\mathbb{Q}$.
\end{conjecture}

Even without the $x+d+\frac{1}{d}$ term, this is not known.

\begin{conjecture}
	The pattern $\{x, x+d, x+\frac{1}{d}\}$ is partition regular over $\mathbb{Q}$.
\end{conjecture}
 
There is hope to use the ideas in this paper to tackle~\Cref{hindman}, where even the simple case $n=2$ is open, and indeed Conjecture 4.1 of \cite{alweiss} essentially states that we \textit{must} use them, or at least that we have to use higher degree polynomials in some capacity, at least if the proof also works over the natural polynomial semi-ring.  We restate Conjecture 4.1 of \cite{alweiss} here, for completeness.  We need the following notation, which also appears in \cite{alweiss}.

\begin{definition} Let $P_d$ be the set of nonzero polynomials of a countable set of variables $x_1, \cdots, x_n, \cdots$ with nonnegative integer coefficients, zero constant term, and degree at most $d$ in each variable. \end{definition}

Conjecture $4.1$ of \cite{alweiss} states the following.

\begin{conjecture} For all $d \in \mathbb{N}$, the pattern $\{x, y, x+y, xy\}$ is not partition regular over $P_d$.
\end{conjecture}

Even the weaker version, that~\Cref{hindman} is false over $P_d$, is not known.  Of course the set $P_{\infty}$ of all such polynomials without the degree bound is not quite the same as $\N$.  It is perhaps possible to have a more ``number-theoretic" proof of~\Cref{hindman} that works over $\N$, where most naturals are not prime, but not over $P_{\infty}$, where most polynomials are irreducible.  A ``purely combinatorial" proof based only on polynomial identities will work just as well over $P_{\infty}$ as over $\N$.  It is worth mentioning that in the definition of $P_{\infty}$ one could equivalently just take the semi-ring $T$ of nonzero polynomials of a single variable with nonnegative coefficients and zero constant term, at least as far as partition regularity of polynomial patterns is concerned.  Partition regularity over $T$ is readily seen to be equivalent to partition regularity of $P_{\infty}$ by standard compactness arguments.

\begin{conjecture} There exists $n$ so that for all $d \in \mathbb{N}$, the pattern $\{\sum_{i \in S}x_i, \prod_{i \in S}x_i\}$, where $S$ ranges over nonempty subsets of $[n]$,  is not partition regular over $P_d$.
\end{conjecture}

The main issue with applying the arguments from this paper straightforwardly to solve~\Cref{hindman} is that we cannot simply ``shift" $ac$ by $b$; this is because $b$ generally will not be divisible by $c$. In other words, to control the color of $P(ac;b)$, we would have to control the color of $P(a;\frac{b}{c})$. However, if $b,b+c,b+2c$ are the same color, we cannot guarantee $c \mid b$.  The $2$-coloring where $3^xy$ for $y \equiv 1 \mod 3$ is red and $3^xy$ for $y \equiv 2 \mod 3$ is blue does not have three consecutive multiples of any integer $c$ in the same color.

Still, we believe it may be possible to get around this sort of obstacle.  We expect that~\Cref{prop:main} is true over $\N$ (and even over $P_{\infty}$).  One idea to prove it is to have many $a_nc_n$ with common difference $b$, so that we can perform the ``shift" update $ac \leftarrow ac+b$, at the cost of also updating and renaming $a$ and $c$.  It seems a little difficult to do so many updates while maintaining the other conditions involving $a$ and $c$ (and all of the other variables), but perhaps it is possible if there are e.g. large sets $S_a$ of possible $a$ and $S_c$ of possible $c$ with some nice properties (and $S_x$ for the other variables $x$).  Another idea is to consider more general polynomial expressions such as the aforementioned $Q(d)(x+P_i(d))$ for general $Q$ with $Q(0)=P_i(0)=0$, rather than restricting ourselves to monic $Q$.  We propose a final conjecture, which arose in conversation with Zach Hunter, and which follows from the $n=2$ case of~\Cref{hindman} but to the best of our knowledge is still unsolved.  In this conjecture, we do not require that $x$ and $y$ are the same color.

\begin{conjecture}
	If $\mathbb{N}$ is finitely colored, there exist $x$ and $y$ so that $x$ and $x+y$ are the same color, and furthermore $y$ and $xy$ are the same color.
\end{conjecture}

\section{Acknowledgments}

Thanks to Noga Alon, Matthew Bowen, Matija Bucic, Timothy Gowers, Neil Hindman, Zach Hunter, Maria Ivan, Noah Kravitz, Imre Leader, Shachar Lovett, Marcin Sabok, GPT-4, and two anonymous readers for helpful comments.

\end{document}